\documentclass{amsart}
\usepackage[margin=1in]{geometry}

\usepackage{color,amssymb, comment, mathrsfs,tikz,upgreek,contour, soul, colonequals, mathdots,fancyvrb, amsmath}
\usepackage[colorlinks, linkcolor={blue}]{hyperref}
\usepackage{tikz-cd}
\usepackage{cleveref}
\usetikzlibrary{cd}
\usetikzlibrary{fit,shapes.geometric}
\usetikzlibrary{matrix}
\usetikzlibrary{arrows}
\usepackage{enumitem}
\usepackage{mathtools}
\usepackage{thmtools}
\usepackage[all]{xy}
\usepackage[mathlines,pagewise]{lineno}
\allowdisplaybreaks

\newtheorem{theorem}{Theorem}[section]
\newtheorem*{maintheorem1}{Main Theorem 1}
\newtheorem*{maintheorem2}{Main Theorem 2}
\newtheorem*{maintheorem3}{Main Theorem 3}
\newtheorem*{corollary2}{Corollary}
\newtheorem{lemma}[theorem]{Lemma}
\newtheorem{proposition}[theorem]{Proposition}
\newtheorem{corollary}[theorem]{Corollary}
\theoremstyle{definition}
\newtheorem*{ack}{Acknowledgements}
\newtheorem*{AI}{AI Disclosure}

\usepackage[only,llbracket,rrbracket,llparenthesis,rrparenthesis]{stmaryrd} 
\usepackage{accsupp}

\theoremstyle{definition}
\newtheorem{definition}[theorem]{Definition}

\theoremstyle{remark}
\newtheorem{remark}[theorem]{Remark}

\numberwithin{equation}{section}

\newcommand{\D}{\textsf{\upshape D}}

\newcommand{\kk}{\Bbbk}

\newcommand{\Ext}{\operatorname{Ext}}

\newcommand{\Hom}{\operatorname{Hom}}

\newcommand{\RHom}{\operatorname{\mathbf{R}Hom}}

\newcommand{\HH}[2]{\operatorname{H}_{#1}(#2)}
\newcommand{\Tor}{\operatorname{Tor}}

\newcommand{\m}{\mathfrak{m}}

\newcommand{\x}{\mathbf{x}}
\newcommand{\depth}{\operatorname{depth}}
\newcommand{\width}{\operatorname{width}}

\newcommand{\ld}{\operatorname{lc.dim}}
\newcommand{\wh}{\widehat}

\newcommand{\thick}{\operatorname{thick}}

\newcommand{\fm}{\mathfrak{m}}

\makeatletter
\DeclareFontEncoding{LS2}{}{\@noaccents}
\makeatother
\DeclareFontSubstitution{LS2}{stix}{m}{n}

\DeclareSymbolFont{largesymbolsstix}{LS2}{stixex}{m}{n}

\DeclareMathDelimiter{\lbrbrak}{\mathopen}{largesymbolsstix}{"EE}{largesymbolsstix}{"14}
\DeclareMathDelimiter{\rbrbrak}{\mathclose}{largesymbolsstix}{"EF}{largesymbolsstix}{"15}

\crefname{diagram}{diagram}{diagrams}
\crefname{diagram}{Diagram}{Diagrams}
\creflabelformat{diagram}{(#1) #2 #3}

\renewcommand{\H}{\mathrm{H}}

\newcommand{\lotimes}{\otimes^{\mathbf{L}}}

\newcommand{\catb}{\sqsubset\mspace{-13mu}\sqsupset}
\newcommand{\catbb}{\sqsupset}
\newcommand{\catba}{\sqsubset}
\newcommand{\Catsub}[3]{{\mathsf{#2}}_{#3}(#1)}
\newcommand{\Catsupsub}[4]{{\mathsf{#2}}^{\text{\upshape #3}}_{#4}(#1)}
\newcommand{\Catsup}[3]{{\mathsf{#2}}^{\text{\upshape #3}}(#1)}
\newcommand{\Db}[1][R]{\Catsub{#1}{D}{\catb}}
\newcommand{\Dfb}[1][R]{\Catsupsub{#1}{D}{f}{\catb}}
\newcommand{\Dbb}[1][R]{\Catsub{#1}{D}{\catbb}}
\newcommand{\Dba}[1][R]{\Catsub{#1}{D}{\catba}}
\newcommand{\Dfba}[1][R]{\Catsupsub{#1}{D}{f}{\catba}}
\newcommand{\Dfbb}[1][R]{\Catsupsub{#1}{D}{f}{\catbb}}
\newcommand{\Df}[1][R]{\Catsup{#1}{D}{f}}

\newcommand{\LLam}[2][\mathfrak{m}]{\nobreak{\mathbf{L}\Lambda^{#1}#2}}
\newcommand{\LLamhat}[2][\mathfrak{m}]{\nobreak{\mathbf{L}\widehat{\Lambda}^{#1}#2}}
\newcommand{\LL}[3][\mathfrak{m}]{\nobreak{\mathbf{L}\Lambda^{#1}_{#2}#3}}
\newcommand{\RGam}[2][\mathfrak{m}]{\nobreak{\mathbf{R}\Gamma_{#1}#2}}
\newcommand{\Tel}{\operatorname{Tel}}
\newcommand{\RG}[3][\mathfrak{m}]{\nobreak{\mathbf{R}\Gamma_{#1}^{#2}#3}}
\newcommand{\amp}{\operatorname{amp}}
\newcommand{\projdim}{\operatorname{projdim}}
\newcommand{\flatdim}{\operatorname{flatdim}}
\newcommand{\injdim}{\operatorname{injdim}}

\newcommand{\HR}{\HH{0}{R}}



\author[Ferraro]{Luigi Ferraro}
\address[Luigi Ferraro]{School of Mathematical and Statistical Sciences \\ University of Texas Rio Grande Valley \\ Edinburg, TX 78539, U.S.A}
\email[]{luigi.ferraro@utrgv.edu}

\author[Nason]{Zach Nason}
\address[Zach Nason]{Department of Mathematics\\
University of Nebraska-Lincoln\\
Lincoln, NE 68503}
\email[]{znason2@huskers.unl.edu}

\begin{document}
\title{Width and amplitude inequalities over DG-rings}

\begin{abstract}
We study width over commutative noetherian local differential graded
(DG) rings with bounded homology. We establish a derived width formula
and prove amplitude inequalities for the derived completion of
derived Hom and tensor products. As applications, we obtain Iversen's
amplitude inequality and a criterion ensuring that regular sequences
on a DG-module are regular on the DG-ring. These applications recover
results for DG-modules with finitely generated homology and also hold
for derived complete DG-modules. Finally, we characterize
Cohen--Macaulay DG-rings by the existence of a bounded derived complete
DG-module of finite injective dimension whose amplitude equals that
of the DG-ring. This provides a version of Bass' conjecture for
derived complete DG-modules.
\end{abstract}

\maketitle
\section{Introduction}

For a bounded below complex over a noetherian local ring, the width of the complex is
the lowest nonvanishing degree of its derived tensor product with the
residue field. It is also the lowest nonvanishing degree of its derived
completion. Christensen, Holm and the first author \cite{amp} used width to establish amplitude inequalities for local (co)homology. In this paper, we study width over differential graded
(DG) rings and prove DG counterparts of these inequalities, with
applications to regular sequences and Cohen--Macaulayness.

Throughout this paper, $R$ is a commutative nonnegative noetherian local DG-ring
with bounded homology, and $(\HR,\m,\kk)$ is its degree-zero homology
ring. We use homological grading and write $\Df$ for the DG-modules
with degreewise finitely generated homology; our remaining conventions
are given in Section~2. The width of a bounded below DG-module $M$ is
$\width_RM=\inf(\kk\lotimes_RM)$. We write $\LLam[\m]{-}$ and
$\RGam[\m]{-}$ for derived completion and torsion,
respectively.

In earlier work \cite[Theorem~4.6]{FN}, the authors proved a DG intersection theorem involving the derived torsion of a derived tensor product.
Our first main result gives a dual bound for the derived completion
of a derived Hom complex.

\begin{maintheorem1}[\Cref{thm:4.6dual}]
Let $0\not\simeq X\in\Dfb$ and $Y\in\Dba$ with
$\width_RY<\infty$ and $\injdim_RY<\infty$. Then
\[
\sup\LLam[\m]{(\RHom_R(X,Y))}
\geq\ld_RX+\width_RY-\depth R.
\]
\end{maintheorem1}

Our next main result consists of two amplitude inequalities for
derived completion. They give DG counterparts of the right-hand
inequalities in \cite[Theorems~2.3 and~2.8]{amp}.

\begin{maintheorem2}[\Cref{thm:4.7dual,cor:ampLLam}]
Let $Y\in\Db$.
\begin{enumerate}
\item If $0\not\simeq X\in\Dfba$, $\width_RY<\infty$, and
$\injdim_RY<\infty$, then
\[
\amp\LLam[\m]{(\RHom_R(X,Y))}\geq\amp\RGam[\m]{(X)}.
\]
\item If $X\in\Df$, $\depth_RY<\infty$, and
$\flatdim_RY<\infty$, then
\[
\amp\LLam[\m]{(X\lotimes_RY)}\geq\amp X.
\]
\end{enumerate}
\end{maintheorem2}

The first inequality recovers \cite[Proposition~4.5]{yang} when both
arguments have bounded, finitely generated homology. Derived completion is
essential to this formulation: it allows us to obtain consequences
under either finite generation or derived completeness hypotheses.

In particular, \Cref{cor:iversen} gives a DG version of Iversen's
amplitude inequality.
\begin{corollary2}[\Cref{cor:iversen}]
Let $Y$ be a nonzero bounded DG $R$-module
with finite flat dimension. If $Y$ has degreewise finitely generated
homology or is derived $\m$-complete, then
\[
\amp(X\lotimes_RY)\geq\amp X\qquad\text{for every }X\in\Df.
\]
\end{corollary2}
For bounded DG-modules with finitely generated homology, this was
proved by Yang \cite[Proposition~4.1]{yang}. Our result also treats
derived complete DG-modules, paralleling
\cite[Corollary~2.9]{amp} over local rings. Under the same hypotheses
on $Y$, the additional equality $\amp Y=\amp R$ implies that every
$Y$-regular sequence is $R$-regular; see \Cref{cor:auslander}.
We call such a $Y$ an Auslander DG-module. This gives a DG analogue
of \cite[Example~2.10]{amp} and includes the finite-homology case
of \cite[Remark~4.2(2)]{yang}.

Our final application concerns the DG version of Bass' conjecture.
Following \cite{ShaulCM}, $R$ is Cohen--Macaulay if
$\amp R=\amp\RGam[\m]{(R)}$. The next result characterizes this
property using either finitely generated homology or derived
completeness.

\begin{samepage}
\begin{maintheorem3}[\Cref{cor:bass}]
The following conditions are equivalent:
\begin{enumerate}
\item $R$ is Cohen--Macaulay.
\item There is a $Y\in\Dfb$ with $\injdim_RY<\infty$ and
$\amp Y=\amp R$.
\item There is a derived $\m$-complete $Y\in\Db$ with
$\injdim_RY<\infty$ and $\amp Y=\amp R$.
\end{enumerate}
\end{maintheorem3}
\end{samepage}

The equivalence of the first two conditions follows from
\cite[Theorem~5.22]{ShaulCM} and \cite[Proposition~4.5]{yang}.
The third condition provides a characterization without finite
generation of homology and is a DG counterpart of
\cite[Corollary~2.2]{amp}. For ordinary local rings, amplitude zero
means that $Y$ is a shift of a module, so this specializes to the
classical Bass characterization and its derived complete version.

The paper is organized as follows. Section~2 collects the background
on DG-modules, Koszul complexes, and local (co)homology.
Section~3 develops the properties of width needed in Section~4,
where we prove the inequalities and their applications.

\begin{ack}
Luigi Ferraro was partly supported by the Simons Foundation grant MPS-TSM-00007849.  
\end{ack}

\begin{AI}
ChatGPT 6 Astra identified a few issues in an earlier version of this paper and suggested fixes. All ideas and proofs in this paper were generated by the authors, and AI was used only to check early drafts.
\end{AI}

\section{Background}\label{sec:background}

We collect the conventions and background used throughout the paper.
As in \cite[Section~2]{FN}, we use homological grading.

\subsection*{DG-rings and DG-modules}
A DG-ring is a graded ring $R=\bigoplus_{i\in\mathbb Z}R_i$ with a
differential $\partial$ of degree $-1$ satisfying $\partial^2=0$ and
\[
\partial(rs)=\partial(r)s+(-1)^{|r|}r\partial(s)
\]
for homogeneous elements $r,s\in R$. We use commutativity in the
graded sense: $rs=(-1)^{|r||s|}sr$, and $r^2=0$ when $|r|$ is odd.
Throughout, $R$ is a commutative nonnegative DG-ring with bounded
homology such that $\HR$ is a noetherian local ring with maximal ideal
$\m$ and residue field $\kk$, and each $\HH{i}{R}$ is finitely
generated over $\HR$. We refer to such an $R$ as a DG-ring.

A DG $R$-module is a graded $R$-module $M$ with a differential of
degree $-1$ satisfying $\partial^2=0$ and
$\partial(rm)=\partial(r)m+(-1)^{|r|}r\partial(m)$.
We regard $\HR$-modules as DG $R$-modules concentrated in degree zero
via the canonical map $R\to\HR$.

\subsection*{The derived category}
We write $\D(R)$ for the derived category obtained from the category
$\operatorname{DG}(R)$ of DG $R$-modules and $R$-linear chain maps by
inverting quasi-isomorphisms, and use $\simeq$
for isomorphisms in $\D(R)$. For $M\in\D(R)$, set
\[
\inf M=\inf\{i\mid\HH{i}{M}\neq0\},\qquad
\sup M=\sup\{i\mid\HH{i}{M}\neq0\},\qquad
\amp M=\sup M-\inf M.
\]
For an acyclic DG-module, these values are $\infty$, $-\infty$, and
$-\infty$, respectively. Our suspension convention is
$\HH{i}{\Sigma^sM}=\HH{i-s}{M}$. We use the full subcategories
\[
\Dba=\{M\in\D(R)\mid\sup M<\infty\},\qquad
\Dbb=\{M\in\D(R)\mid\inf M>-\infty\},
\]
and set $\Db=\Dba\cap\Dbb$. The subcategory $\Df$ consists of the
DG-modules whose homology modules are finitely generated over
$\HR$ in every degree. Intersecting $\Df$ with $\Dba$, $\Dbb$,
and $\Db$ gives $\Dfba$, $\Dfbb$, and $\Dfb$, respectively.
In particular, membership in $\Df$ does not impose boundedness.

The derived tensor product and derived Hom are denoted by
$-\lotimes_R-$ and $\RHom_R(-,-)$. Thus
\[
\Tor_i^R(M,N)=\HH{i}{M\lotimes_RN},\qquad
\Ext_R^i(M,N)=\HH{-i}{\RHom_R(M,N)}.
\]
For $M\in\D(R)$, the notation $\thick_RM$ denotes the smallest
full triangulated subcategory of $\D(R)$ containing $M$ and closed
under direct summands. An object $P \in \D(R)$ is compact if the functor $\RHom_R(P, -)$ commutes with arbitrary coproducts. The compact objects of $\D(R)$ are precisely the objects of $\thick_RR$, also called perfect DG
modules; see \cite[Theorem~5.20 and Corollary~5.21]{tilting}. When $R$ is a noetherian local ring, the perfect objects are the complexes with bounded and degreewise finitely generated homology with finite projective dimension.
\subsection*{Homological dimensions}
For $M\in\D(R)$, we use the following definitions:
\begin{align*}
\projdim_RM
&=\inf\{n\in\mathbb Z\mid\Ext_R^i(M,N)=0
  \text{ for all }N\in\Db\text{ and }i>n-\inf N\},\\
\injdim_RM
&=\inf\{n\in\mathbb Z\mid\Ext_R^i(N,M)=0
  \text{ for all }N\in\Db\text{ and }i>n+\sup N\},\\
\flatdim_RM
&=\inf\{n\in\mathbb Z\mid\Tor_i^R(N,M)=0
  \text{ for all }N\in\Db\text{ and }i>n+\sup N\}.
\end{align*}
These agree with the usual dimensions over an ordinary ring; see
\cite[Definition~2.1]{ShaulFinDim}, with the grading reversed.

\subsection*{Depth and local cohomology Krull dimension}
For $M\in\Dba$ and an ideal $I\subseteq\HR$, set
\[
\depth_R(I,M)=-\sup\RHom_R(\HR/I,M),\qquad
\depth_RM=\depth_R(\m,M).
\]
For $M\in\Dbb$, its local cohomology Krull dimension is
\[
\ld_RM=\sup_{i\in\mathbb Z}
\{\dim_{\HR}\HH{i}{M}-i\},
\]
where $\dim_{\HR}0=-\infty$; see
\cite[Definitions~2.1 and~3.1]{ShaulCM} and
\cite[Definition~3.1]{ShaulKoszul}.

\subsection*{Koszul DG-modules}
Let $\x=x_1,\ldots,x_n$ be a sequence in $\HR$, and choose lifts
$\widetilde\x=\widetilde x_1,\ldots,\widetilde x_n$ in $R_0$.
We write
\[
K^R(\x)=R\otimes_{R_0}K^{R_0}(\widetilde\x),
\]
where the Koszul complex on the right is in degrees $0,\ldots,n$.
Its isomorphism class in $\D(R)$ is independent of the lifts; see
\cite[Definition~2.4 and Proposition~2.6]{ShaulKoszul}.
We also write $K(\x)$ when the DG-ring is understood.

When $\x$ generates $\m$, write $K=K(\x)$. The DG-ring $K$ is
noetherian, has bounded homology, and satisfies $\HH{0}{K}=\kk$. Applying
\cite[Corollary~4.3]{ShaulLifting} to $K$ gives
$K\in\thick_K\kk$. Restriction of scalars along $R\to K$
preserves thick constructions by \cite[Lemma~4.1]{ShaulLifting},
so
\[
K\in\thick_R\kk.
\]
We will also use the following consequence of Nakayama's lemma:
for every $M\in\Df$ and every $i\in\mathbb Z$,
\[
\HH{i}{K\lotimes_RM}=0\quad\Longrightarrow\quad\HH{i}{M}=0.
\]
Indeed, for a single element $x\in\m$, the triangle
$M\xrightarrow{x}M\to K(x)\lotimes_RM\to\Sigma M$ gives an injection
\[
\HH{i}{M}/x\HH{i}{M}\longrightarrow\HH{i}{K(x)\lotimes_RM},
\]
and Nakayama's lemma applies. Iterating this argument proves the
claim, since tensoring with a finite Koszul DG-module preserves
degreewise finiteness of homology. Moreover,
$K\lotimes_RM\in\thick_R(\kk\lotimes_RM)$. It follows that
\[
M\not\simeq0\ \Longrightarrow\ \kk\lotimes_RM\not\simeq0,
\qquad
\kk\lotimes_RM\in\Dbb\ \Longrightarrow\ M\in\Dbb.
\]
No boundedness assumption on $M\in\Df$ is needed for these assertions.

\subsection*{Local cohomology and derived completion}
Let $I\subseteq\HR$ be an ideal, and choose a finite sequence
$\widetilde\x$ in $R_0$ whose image generates $I$.
For a ring $S$ and $x\in S$, let
$\check C_S(x)=[S\to S_x]$, in degrees $0$ and $-1$.
The augmented \v Cech complex on a sequence is the tensor product
of these complexes; we denote the one on $\widetilde\x$ over
$R_0$ by $\check C(\widetilde\x)$.

For completeness, the one-element telescope complex is
$\Tel(S;x)=[F\xrightarrow{d}F]$, also in degrees $0$ and $-1$,
where $F=\bigoplus_{j\geq0}Se_j$ and
\[
d(e_0)=e_0,\qquad d(e_j)=e_{j-1}-xe_j\quad(j\geq1).
\]
For a sequence, take the tensor product of the corresponding
one-element telescope complexes. Thus $\Tel(R_0;\widetilde\x)$
is a bounded complex of free $R_0$-modules, quasi-isomorphic to
$\check C(\widetilde\x)$; see \cite[Section~7.1.3]{injective}.

The derived $I$-torsion and derived $I$-completion functors are
computed by
\begin{align*}
\RGam[I]{(M)}
&\simeq\check C(\widetilde\x)\lotimes_{R_0}M
 \simeq\Tel(R_0;\widetilde\x)\otimes_{R_0}M,\\
\LLam[I]{(M)}
&\simeq\RHom_{R_0}(\check C(\widetilde\x),M)
 \simeq\Hom_{R_0}(\Tel(R_0;\widetilde\x),M)
\end{align*}
for $M\in\D(R)$; see
\cite[Proposition~2.4 and Corollary~2.13]{ShaulCompletion}.
These constructions are independent of the chosen lifts and
generators. We include the ring in the notation when needed, writing
$\RG[I]{R}{(M)}$ and $\LL[I]{R}{(M)}$.

The DG-module $M$ is derived $I$-torsion if the canonical morphism
$\RGam[I]{(M)}\to M$ is an isomorphism, and derived $I$-complete if
the canonical morphism $M\to\LLam[I]{(M)}$ is an isomorphism.
The first condition is equivalent to each $\HH{i}{M}$ being
$I$-torsion; see \cite[Proposition~2.9]{ShaulCompletion}.

\subsection*{The completed DG-ring}
We write $\wh R=\LLam[\m]{(R)}$ for the derived $\m$-adic completion
of $R$, equipped with its DG-ring structure as in
\cite[Section~4]{ShaulCompletion}. It is again a noetherian local
DG-ring with bounded homology, and
\[
\HH{0}{\wh R}\cong\wh{\HR},\qquad
\wh\m=\m\HH{0}{\wh R},\qquad
\flatdim_R\wh R=0;
\]
see \cite[Proposition~1.7]{ShaulCM} and
\cite[Corollary~4.6]{ShaulHomDim}.
Here $\wh{\HR}$ is the ordinary $\m$-adic completion of $\HR$,
and $\wh\m$ is its maximal ideal. In particular, the induced map
$\HR\to\HH{0}{\wh R}$ is faithfully flat.

We distinguish the $R$-valued functor $\LLam[\m]{}$ from its
$\wh R$-valued version
$\LLamhat[\m]{}\colon\D(R)\to\D(\wh R)$.
By \cite[Propositions~5.4 and~5.5]{ShaulCompletion},
\[
\LLamhat[\m]{(M)}
\simeq\RHom_R(\wh R,\LLam[\m]{(M)})
\quad\text{in }\D(\wh R),
\]
and restriction of scalars sends $\LLamhat[\m]{(M)}$ to
$\LLam[\m]{(M)}$. Completion of a DG $\wh R$-module at
$\wh\m$ is denoted by $\LL[\wh\m]{\wh R}{}$.

\subsection*{Injective and dualizing DG-modules}
Let $E_{\HR}$ be the injective hull of $\kk$ over $\HR$.
Throughout, $E$ denotes the indecomposable injective DG $R$-module
corresponding to $\m$ under the classification in
\cite[Section~7.2]{injective}. Thus $\HH{0}{E}\cong E_{\HR}$ and
\[
\RHom_R(\HR,E)\simeq E_{\HR}\quad\text{in }\D(\HR).
\]
For every $M\in\D(R)$ and $i\in\mathbb Z$,
\[
\HH{i}{\RHom_R(M,E)}
\cong\Hom_{\HR}(\HH{-i}{M},E_{\HR});
\]
see \cite[Theorem~4.10]{injective}.
Since $E_{\HR}$ is an injective cogenerator, $\RHom_R(-,E)$ detects
nonzero objects and reverses the degrees in which homology is
nonzero. We use $E_{\wh R}$ for the analogous injective DG-module
over $\wh R$, corresponding to $\wh\m$.

A dualizing DG-module for $R$ is an object $D\in\Dfb$ with
$\injdim_RD<\infty$ for which the homothety morphism
$R\to\RHom_R(D,D)$ is an isomorphism; see
\cite[Definition~7.1 and Corollary~7.3]{tilting}.
We call $D$ normalized if $\RHom_R(\kk,D)\simeq\kk$.
Every dualizing DG-module admits a unique normalizing shift;
see \cite[Section~7.2]{injective}.
The completed DG-ring $\wh R$ always has a dualizing DG-module
by \cite[Proposition~7.21]{injective}.

\section{Width}

In this section, we study width for DG-modules and show that its familiar properties over local rings extend to the DG setting. These properties will be used in the next section to prove the main results of this paper.

We recall that by \cite[Definition 5.1]{width}, if $I$ is an ideal of $\HR$ and $M\in\Dbb$, then the $I$-width of $M$ is
\[
\width_R(I,M)=\inf (\HR/I\lotimes_R M).
\]
We set
\[
\width_RM=\width_R(\m,M).
\]

\begin{remark}\label{rmk:yang}
It is proved in \cite[Theorem 5.2]{width} that if $M\in\Dbb$, then
\begin{equation}\label{eq:widthBaseChange}
\width_R(I,M)=\width_{\HR}(I,\HR\lotimes_RM),
\end{equation}
indeed this easily follows from the following equalities
\begin{align*}
\width_R(I,M)&=\inf (\HR/I\lotimes_RM)\\
&=\inf (\HR/I\lotimes_{\HR}(\HR\lotimes_R M))\\
&=\width_{\HR}(I,\HR\lotimes_RM).
\end{align*}
\end{remark}

We first remove a boundedness hypothesis from \cite[Lemma 4.2]{ShaulHomDim}.
\begin{lemma}\label{lem:4.2}
Let $I$ be an ideal of $\HR$ and let $M\in\Dbb$. Then $\LLam[I]{M}\in\Dbb$, and there is a natural isomorphism
\[
\HR\lotimes_R\LL[I]{R}{(M)}\simeq\LL[I]{\HR}{(\HR\lotimes_RM)}
\]
in $\D(\HR)$.
\end{lemma}
\begin{proof}
Choose a finite sequence $\x$ in $R_0$ whose image
$\overline{\x}$ in $\HR$ generates $I$, and set
$T=R\otimes_{R_0}\Tel(R_0;\x)$.
By \cite[Corollary 2.3(i)]{ShaulFinDim} $\projdim_RT<\infty$.
The telescope description of completion gives
\[
\LL[I]{R}{(M)}\simeq\Hom_{R_0}(\Tel(R_0;\x),M)
\simeq\RHom_R(T,M);
\]
see \cite[(2.2)]{ShaulHomDim}.
Taking a semifree resolution of $M$ that is bounded below,
as in \cite[Proposition 1.19(1)]{tilting}, shows from the first
description that $\LL[I]{R}{(M)}$ is bounded below.

We now apply the tensor-evaluation isomorphism
\cite[Lemma 2.7(iv)]{ShaulFinDim} with its three arguments equal to
$T$, $M$, and $\HR$, respectively. Its hypotheses hold because
$\projdim_RT<\infty$, $M\in\Dbb$, and $\HR\in\Dfbb$.
Consequently,
\begin{align*}
\HR\lotimes_R\LL[I]{R}{(M)}
&\simeq\RHom_R(T,M)\lotimes_R\HR\\
&\simeq\RHom_R(T,M\lotimes_R\HR)\\
&\simeq\RHom_{\HR}(\HR\lotimes_RT,\HR\lotimes_RM)\\
&\simeq\LL[I]{\HR}{(\HR\lotimes_RM)}.
\end{align*}
The base change identity
$\HR\lotimes_RT\simeq\Tel(\HR;\overline{\x})$
and the telescope description over $\HR$ give the last isomorphism.
\end{proof}

The following result is a DG version of \cite[Theorem 14.4.8]{LarsBook} and provides several alternative ways of computing the width of a DG-module.

\begin{theorem}\label{thm:width}
Let $I$ be an ideal of $\HR$ generated by $\x=x_1,\ldots,x_n$ and $M\in\Dbb$. Then
\begin{align*}
\width_R(I,M)&=\inf\LLam[I]{(M)}\\
&=\inf (K(\x)\lotimes_RM)\\
&=n+\inf\RHom_R(K(\x),M).
\end{align*}
\end{theorem}
\begin{proof} 
We first recall that by \Cref{rmk:yang} one has
\[
\width_R(I,M)=\width_{\HR}(I,\HR\lotimes_RM).
\]
We start with the first equality
\begin{align*}
\width_{\HR}(I,\HR\lotimes_RM)&=\inf\LL[I]{\HR}{(\HR\lotimes_RM)}\\
&=\inf (\HR\lotimes_R\LL[I]{R}{M})\\
&=\inf\LL[I]{R}{(M)},
\end{align*}
where the first equality follows from \cite[Theorem 14.4.8]{LarsBook}, the second from \Cref{lem:4.2}, and the last from \cite[(1.3)]{ShaulCM}.

We now prove that $\width_R(I,M)=\inf (K(\x)\lotimes_RM)$:
\begin{align*}
\width_{\HR}(I,\HR\lotimes_RM)&=\inf (K^{\HR}(\x)\lotimes_{\HR}\HR\lotimes_RM)\\
&=\inf (K^{\HR}(\x)\lotimes_RM)\\
&=\inf (K^R(\x)\lotimes_R\HR\lotimes_RM)\\
&=\inf (K^R(\x)\lotimes_RM),
\end{align*}
where the first equality follows from \cite[Theorem 14.4.8]{LarsBook}, the second is tensor cancellation, the third equality comes from \cite[Proposition 2.9]{ShaulKoszul} and the fourth from \cite[(1.3)]{ShaulCM}.

Now the final equality
\begin{align*}
\width_{\HR}(I,\HR\lotimes_RM)&=n+\inf\RHom_{\HR}(K^{\HR}(\x),\HR\lotimes_RM)\\
&=n+\inf\RHom_{\HR}(K^{R}(\x)\lotimes_R\HR,\HR\lotimes_RM)\\
&=n+\inf\RHom_{R}(K^{R}(\x),\RHom_{\HR}(\HR,\HR\lotimes_RM))\\
&=n+\inf\RHom_{R}(K^{R}(\x),\HR\lotimes_RM)\\
&=n+\inf\RHom_{R}(K^{R}(\x),M)\lotimes_R\HR\\
&=n+\inf\RHom_{R}(K^{R}(\x),M),
\end{align*}
where the first equality follows from \cite[Theorem 14.4.8]{LarsBook}, the second is \cite[Proposition 2.9]{ShaulKoszul}, the third is tensor-hom adjunction, the fourth is hom-cancellation, the fifth follows from the tensor-evaluation part~(i) of \cite[Lemma~2.7]{ShaulFinDim}, and the final equality from \cite[(1.3)]{ShaulCM}.
\end{proof}

The next result is a DG version of \cite[Proposition 14.3.27]{LarsBook}. Its main purpose is to show that depth and width are simultaneously finite. 

\begin{proposition}\label{prp:finiteDepth} Let $M\in\Db$ and let $I$ be an ideal of $\HR$ generated by $\x=x_1,\ldots,x_n$. Then the following are equivalent
\begin{enumerate}
\item $\depth_R(I,M)<\infty$.
\item $K(\x)\lotimes_RM$ is not acyclic.
\item $\width_R(I,M)<\infty$. 
\end{enumerate}
\end{proposition}
\begin{proof}
By \cite[Proposition 3.17]{ShaulKoszul} (1) is equivalent to $\sup (K(\x)\lotimes_RM)>-\infty$, i.e. $K(\x)\lotimes_RM$ not acyclic. By \Cref{thm:width} (3) is also equivalent to (2). 
\end{proof}

The next result is a DG version of \cite[Theorem 14.3.28]{LarsBook} and relates the width of a DG-module to its infimum.

\begin{theorem}\label{thm:widthInf}
Let $I$ be a proper ideal of $\HR$ and $M\in\Dbb$, then $\width_R(I,M)\geq\inf M$. Moreover if $M$ is derived $I$-complete or if $M\in\Dfbb[R]$, then $\width_R(I,M)=\inf M$.
\end{theorem}

\begin{proof}
The inequality follows from 
\begin{align*}
\width_R(I,M)&=\width_{\HR}(I,\HR\lotimes_RM)\\
&\geq\inf (\HR\lotimes_RM)\\
&=\inf M,
\end{align*}
where the first equality is from \Cref{rmk:yang}, the inequality from \cite[Theorem 14.3.28]{LarsBook}, and the last equality from \cite[(1.3)]{ShaulCM}.

If $M$ is derived $I$-complete in $\D(R)$, then $\HR\lotimes_RM$ is derived $I$-complete in $\D(\HR)$, this follows directly from \Cref{lem:4.2}. Similarly, if $M\in\Dfbb$, then $\HR\lotimes_RM\in\Dfbb[\HR]$ by an argument similar to the proof of \cite[Theorem 2.13]{tilting}. Therefore, under these hypotheses we can argue as in the previous paragraph to get
\begin{align*}
\width_R(I,M)&=\width_{\HR}(I,\HR\lotimes_RM)\\
&=\inf (\HR\lotimes_RM)\\
&=\inf M.\qedhere
\end{align*}
\end{proof}

The next result is a DG version of \cite[Proposition 16.2.22]{LarsBook} and shows the behavior of depth and width under Matlis duality.

\begin{proposition}\label{prp:widthDual}
Let $X\in\D(R)$.
\begin{enumerate}
\item If $X\in\Dba$, then $\depth_RX=\width_R\RHom_R(X,E)$.
\item If $X\in\Dbb$, then $\width_RX=\depth_R\RHom_R(X,E)$.
\end{enumerate}
\end{proposition}
\begin{proof}

\begin{enumerate}
\item The assertion follows from the string of equalities below
\begin{align*}
\width_R\RHom_R(X,E)&=\width_{\HR}\RHom_R(X,E)\lotimes_R\HR\\
&=\width_{\HR}\RHom_R(\RHom_R(\HR,X),E)\\
&=\width_{\HR}\RHom_{\HR}(\RHom_R(\HR,X),\RHom_R(\HR,E))\\
&=\width_{\HR}\RHom_{\HR}(\RHom_R(\HR,X),E_{\HR})\\
&=\depth_{\HR}\RHom_R(\HR,X)\\
&=\depth_RX,
\end{align*}
where the first equality is from \eqref{eq:widthBaseChange}, the second from the hom-evaluation part~(i) of \cite[Lemma~2.7]{ShaulFinDim}, the third one is tensor-hom adjunction, the fourth one is from the definition of $E$ (see \cite{injective}), the fifth from \cite[Proposition 16.2.22]{LarsBook}, and the last one is standard.

\item The assertion follows from the string of equalities below
\begin{align*}
\depth_R\RHom_R(X,E)&=\depth_{\HR}\RHom_R(\HR,\RHom_R(X,E))\\
&=\depth_{\HR}\RHom_R(X,\RHom_R(\HR,E))\\
&=\depth_{\HR}\RHom_R(X,E_{\HR})\\
&=\depth_{\HR}\RHom_{\HR}(X\lotimes_R\HR,E_{\HR})\\
&=\width_{\HR}(X\lotimes_R\HR)\\
&=\width_RX,
\end{align*}
where the first equality follows from \cite[Proposition 3.2]{ShaulCM}, the second equality follows by applying tensor-Hom adjunction twice, the third equality follows from the definition of $E$ (see \cite{injective}), the fourth is tensor-Hom adjunction, the fifth follows from \cite[Proposition 16.2.22]{LarsBook}, and the sixth from \eqref{eq:widthBaseChange}. \qedhere
\end{enumerate}
\end{proof}

The next result is a DG version of \cite[Proposition 18.2.30]{LarsBook} and shows the behavior of depth and width under Grothendieck duality.

\begin{proposition}\label{prp:widthGDual}
Let $X\in\D(R)$ and $D$ a normalized dualizing DG $R$-module.
\begin{enumerate}
\item If $X\in\Dba$, then $\depth_RX=\width_R\RHom_R(X,D)$.
\item If $X\in\Dbb$, then $\width_RX=\depth_R\RHom_R(X,D)$.
\end{enumerate}
\end{proposition}
\begin{proof}
We first note that by \cite[Proposition 7.5(1)]{tilting} $\RHom_R(\HR,D)$ is a dualizing complex for $\HR$ and it is normalized by derived adjunction, since
\[
\RHom_{\HR}(\kk,\RHom_R(\HR,D))\simeq\RHom_R(\kk,D)\simeq\kk.
\]
The same proof as \Cref{prp:widthDual} works now by invoking \cite[Proposition 18.2.30]{LarsBook} instead of \cite[Proposition 16.2.22]{LarsBook}.
\end{proof}
Next we prove the Derived Width Formula for DG-rings, which is dual to the Derived Depth Formula proved by Yang in \cite[Lemma 2.5]{yang}. Over local rings it is proved in \cite[Theorem 16.3.9]{LarsBook}.

\begin{theorem}\label{thm:dwf}
Let $M\in\Dfb$ and $N\in\Db$ such that $\RHom_R(M,N)\in\Db$. If $\flatdim_RM<\infty$ or $\injdim_RN<\infty$, then
\[
\width_R\RHom_R(M,N)=\depth_RM+\width_RN-\depth R.
\]
\end{theorem}

\begin{proof}
The asserted equality follows from the computation below
\begin{align*}
\width_R\RHom_R(M,N)&=\depth_R\RHom_R(\RHom_R(M,N),E)\\
&=\depth_R(M\lotimes_R\RHom_R(N,E))\\
&=\depth_RM+\depth_R\RHom_R(N,E)-\depth R\\
&=\depth_RM+\width_RN-\depth R,
\end{align*}
where the first equality follows from \Cref{prp:widthDual}(2), the second from the hom-evaluation part~(i) of \cite[Lemma~2.7]{ShaulFinDim}, the third from \cite[Lemma 2.5]{yang} and \cite[Lemma 7.12]{ShaulFinDim}, and the fourth from \Cref{prp:widthDual}(2). 
\end{proof}

The next result shows that derived completion commutes with the functor $\RHom_R(\HR,-)$. This compatibility allows us to transfer bounds for local homology over local rings to the DG setting. It is a special case of \cite[Proposition 2.16]{ShaulCompletion}, but we provide a direct proof.

\begin{theorem}\label{thm:LLamCobaseChange}
Let $X \in \Dba$. Then the following isomorphism holds:
\begin{equation*}
\LL[\m\HR]{\HR}{(\RHom_R(\HR, X))} \simeq \RHom_R(\HR, \LL[\m]{R}{(X)}).
\end{equation*}
\end{theorem}
\begin{proof}
Let $\mathbf{x}$ be a finite sequence of elements in $R_0$ such that its image in $\HR$ generates $\m$, and let $\overline{\mathbf{x}}$ be its image in $\HR$. By the base change property of the telescope complex, we have that $\Tel(\HR; \overline{\mathbf{x}}) \cong \Tel(R_0; \mathbf{x}) \otimes_{R_0} \HR$. The $\HR$-complex $\LL[\m\HR]{\HR}{(\RHom_R(\HR, X))}$ is equivalent to the complex $\Hom_{\HR}(\Tel(\HR; \overline{\mathbf{x}}), \RHom_R(\HR, X))$ by \cite[Theorem 13.1.15]{LarsBook}, and through adjunction and the unitor this is equivalent to $\Hom_{R_0}(\Tel(R_0; \mathbf{x}), \RHom_R(\HR, X))$. Note that $\Tel(\HR; \overline{\mathbf{x}})$ is a semi-free $\HR$-complex and $\Tel(R_0; \mathbf{x})$ is a semi-free $R_0$-complex. We now have the isomorphisms
\begin{align*}
\Hom_{R_0}(\Tel(R_0; \mathbf{x}), \RHom_R(\HR, X)) &\simeq \RHom_R(\HR, \Hom_{R_0}(\Tel(R_0; \mathbf{x}), X)) \\
    &\simeq \RHom_R(\HR, \LL[\m]{R}{(X)})
\end{align*}
where the first isomorphism is swap and the second isomorphism is from \cite[Proposition 2.4]{ShaulCompletion}. This gives the desired result.
\end{proof}

The next result provides bounds on the local (co)homology of a complex. It is a DG version of \cite[Proposition 16.2.34]{LarsBook}.

\begin{theorem}\label{thm:bounds}
Let $M\in\D(R)$, then
\begin{enumerate}
\item if $M\in\Dbb$, then $-\inf\RGam[\m]{(M)}\leq\ld R-\width_RM$,
\item if $M\in\Dba$, then $\sup\LLam[\m]{(M)}\leq\ld R-\depth_RM$.
\end{enumerate}
\end{theorem}
\begin{proof}
\begin{enumerate}
\item The asserted inequality follows from the string of (in)equalities below
\begin{align*}
-\inf\RG[\m]{R}{(M)}&=-\inf\HR\lotimes_R\RG[\m]{R}{(M)}\\
&=-\inf\RG[\m\HR]{\HR}{(\HR\lotimes_RM)}\\
&\leq\dim\HR-\width_{\HR}(\HR\lotimes_RM)\\
&=\ld R-\width_RM,
\end{align*}
where the first equality follows from  \cite[(1.3)]{ShaulCM}, the second equality follows from \cite[Proposition 2.11]{ShaulCM}, the inequality follows from \cite[Proposition 16.2.34]{LarsBook}, and the last equality from \cite[Remark 2.3]{ShaulCM} and \eqref{eq:widthBaseChange}.
\item The asserted inequality follows from the string of (in)equalities below
\begin{align*}
\sup\LL[\m]{R}{(M)}&=\sup\RHom_R(\HR,\LL[\m]{R}{(M)})\\
&=\sup\LL[\m\HR]{\HR}{(\RHom_R(\HR,M))}\\
&\leq\dim\HR-\depth_{\HR}\RHom_R(\HR,M)\\
&=\ld R-\depth_RM,
\end{align*}
where the first equality follows from \cite[(1.2)]{ShaulCM}, the second equality from \Cref{thm:LLamCobaseChange}, the inequality from \cite[Proposition 16.2.34]{LarsBook}, and the last equality from \cite[Remark 2.3]{ShaulCM} and \cite[Proposition 3.2]{ShaulCM}.\qedhere
\end{enumerate}
\end{proof}
The next corollary was proved over commutative local noetherian rings by Strooker in \cite[Corollary 6.1.10]{Strooker}, who credits Bartijn with observing it in his thesis. It can also be found in \cite[Corollary 16.2.35]{LarsBook}.
\begin{corollary}
Let $M\in\Db$ with $\depth_RM<\infty$. Then
\[
\depth_RM+\width_RM\leq\ld R.
\]
\end{corollary}
\begin{proof}
By \Cref{prp:finiteDepth} $\width_RM<\infty$. Now we observe that
\begin{align*}
\depth_RM+\width_RM&=-\sup\RGam[\m]{(M)}+\width_RM\\
&\leq-\inf\RGam[\m]{(M)}+\width_RM\\
&\leq\ld R,
\end{align*}
where the equality follows from \cite[Proposition 3.3]{ShaulCM}, the first inequality is trivial, and the last inequality follows from \Cref{thm:bounds}.
\end{proof}

The next result, which will be used later in the paper, shows that derived completion does not change the amplitude of a bounded complex with finitely generated homology.
\begin{proposition}\label{prp:ampCompletion}
If $X\in\Dfb$, then 
\[
\inf X=\inf\LLam[\m]{(X)},\quad\text{and}\quad\sup X=\sup\LLam[\m]{(X)}.
\]
\end{proposition}

\begin{proof}
We first show that $\inf X=\inf\LLam[\m]{(X)}$, indeed
\begin{align*}
\inf X&=\inf(\HR\lotimes_RX)\\
&=\inf \LL[\m\HR]{\HR}{(\HR\lotimes_RX)}\\
&=\inf (\HR\lotimes_R\LL[\m]{R}{(X)})\\
&=\inf\LL[\m]{R}{(X)},
\end{align*}
where the first equality follows from \cite[(1.3)]{ShaulCM}, the second follows since the functor $\LL[\m\HR]{\HR}{(-)}$ is equivalent to the functor $-\lotimes_{\HR}\widehat{\HR}$ on complexes with finitely generated homology, and $\widehat{\HR}$ is faithfully flat over $\HR$ \cite[Proposition 16.1.10]{LarsBook}, the third equality follows from \cite[Lemma 4.2]{ShaulHomDim}, and the last one from \cite[(1.3)]{ShaulCM}.

For the second equality we observe
\begin{align*}
\sup X&=\sup\RHom_R(\HR,X)\\
&=\sup\LL[\m\HR]{\HR}{(\RHom_R(\HR,X))}\\
&=\sup\RHom_R(\HR,\LL[\m]{R}{(X)})\\
&=\sup\LL[\m]{R}{(X)},
\end{align*}
where the first equality follows from \cite[(1.2)]{ShaulCM}, the second follows similarly as above, the third one from \Cref{thm:LLamCobaseChange}, and the last one from \cite[(1.2)]{ShaulCM}.
\end{proof}

The next result, which will be used later in the paper, is dual to \cite[Corollary 4.3]{FN} and connects the width of a DG $R$-module with the width of its cobase change over $\HR$.
\begin{proposition}\label{prp:widthCobaseChange}
If $X\in\Db$ with $\injdim_RX<\infty$, then
\[
\width_RX=\width_{\HR}\RHom_R(\HR,X)+\depth R-\depth \HR.
\]
\end{proposition}
\begin{proof}
The desired equality follows from the string of equalities below
\begin{align*}
\width_RX&=\depth_R\RHom_R(X,E)\\
&=\depth_{\HR}(\HR\lotimes_R\RHom_R(X,E))+\depth R-\depth\HR\\
&=\depth_{\HR}\RHom_R(\RHom_R(\HR,X),E)+\depth R-\depth \HR\\
&=\depth_{\HR}\RHom_{\HR}(\RHom_R(\HR,X),E_{\HR})+\depth R-\depth \HR\\
&=\width_{\HR}\RHom_R(\HR,X)+\depth R-\depth \HR.
\end{align*}
The first equality follows from \Cref{prp:widthDual}, the second from \cite[Lemma 7.12]{ShaulFinDim} and \cite[Corollary 4.3]{FN}, the third from the hom-evaluation part~(i) of \cite[Lemma~2.7]{ShaulFinDim}, the fourth from tensor-hom adjunction, and the last from another application of \Cref{prp:widthDual}.
\end{proof}
The following result is a DG version of \cite[Proposition 14.4.1]{LarsBook}.
\begin{proposition}\label{prp:14.4.1}
Let $I$ be a proper ideal of $\HR$. Then,
\begin{enumerate}
\item If $M\in\Dbb$, then 
\[
\width_R(I,\LLam[I]{(M)})=\width_R(I,M)=\width_R(I,\RGam[I]{(M)}).
\]
\item If $M\in\Dba$, then
\[
\depth_R(I,\LLam[I]{(M)})=\depth_R(I,M)=\depth_R(I,\RGam[I]{(M)}).
\]
\end{enumerate}
\end{proposition}
\begin{proof}Let $K$ be the Koszul complex on a sequence of $n$ elements that
generates $I$. Since $K$ is derived $I$-torsion, it follows from \cite[Proposition 2.11]{ShaulCompletion} that there are isomorphisms
\[
\RHom_R(K,\LLam[I]{(M)})\simeq\RHom_R(K,M)
\simeq\RHom_R(K,\RGam[I]{(M)}).
\]
By \cite[Proposition 2.10]{ShaulKoszul} $K\in\Dfb$ and $\projdim_R K<\infty$. Thus, for every
$N\in\D(R)$, tensor evaluation gives
\[
\RHom_R(K,R)\lotimes_R N
\xrightarrow{\ \simeq\ }\RHom_R(K,N)
\]
by the tensor-evaluation part~(i) of \cite[Lemma~2.7]{ShaulFinDim}.
Moreover, self-duality gives $\RHom_R(K,R)\simeq\Sigma^{-n}K$ by
\cite[Proposition 2.12]{ShaulKoszul}. Consequently,
\[
\RHom_R(K,N)\simeq\Sigma^{-n}(K\lotimes_R N).
\]
Applying this to the three terms above and shifting by $n$ yields
\[
K\lotimes_R\LLam[I]{(M)}\simeq K\lotimes_R M
\simeq K\lotimes_R\RGam[I]{(M)}.
\]
The width equalities now follow from \Cref{thm:width},
which gives $\width_R(I,M)=\inf(K\lotimes_R M)$.
The depth equalities follow from \cite[Proposition 3.17]{ShaulKoszul}.
\end{proof}

The next result, together with \Cref{thm:width}, provides a full DG version of \cite[Theorem 16.2.3]{LarsBook}.

\begin{proposition}
Let $I$ be a proper ideal of $\HR$ and $M\in\Dbb$, then
\[
\width_RM=\width_R\LLam[I]{(M)}=\width_R\RGam[I]{(M)}.
\]
\end{proposition}
\begin{proof}
Since $I\subseteq\m$, every
derived $\m$-torsion DG-module is derived $I$-torsion, and
\cite[Proposition 2.9]{ShaulCompletion} gives for every $X\in\D(R)$
\[
\RGam[I]{(\RGam[\m]{(X)})}\simeq\RGam[\m]{(X)}.
\]
The tensor description of derived torsion in \cite[Proposition 2.4]{ShaulCompletion} shows
that the two derived torsion functors commute. Consequently,
\[
\RGam[\m]{(\RGam[I]{(X)})}\simeq\RGam[\m]{(X)}.
\]

Applying \cite[Proposition 2.11]{ShaulCompletion}, we obtain
\begin{align*}
\LLam[\m]{(\LLam[I]{(M)})}
&\simeq\RHom_R(\RGam[\m]{(R)},\LLam[I]{(M)})\\
&\simeq\RHom_R(\RGam[I]{(\RGam[\m]{(R)})},M)\\
&\simeq\RHom_R(\RGam[\m]{(R)},M)\\
&\simeq\LLam[\m]{(M)}.
\end{align*}

The functors in question preserve boundedness below by
\cite[Proposition 2.4]{ShaulCompletion}. We may therefore apply
\Cref{prp:14.4.1} at $\m$ to obtain
\begin{align*}
\width_R\LLam[I]{(M)}
&=\width_R\LLam[\m]{(\LLam[I]{(M)})}
=\width_R\LLam[\m]{(M)}
=\width_R M,\\
\width_R\RGam[I]{(M)}
&=\width_R\RGam[\m]{(\RGam[I]{(M)})}
=\width_R\RGam[\m]{(M)}
=\width_R M.\qedhere
\end{align*}
\end{proof}

A similar proof shows
\begin{proposition}\label{prp:depthLLam}
Let $I$ be a proper ideal of $\HR$ and $M\in\Dba$, then
\[
\depth_RM=\depth_R\LLam[I]{(M)}=\depth_R\RGam[I]{(M)}.
\]
\end{proposition}
\section{Amplitude inequalities}

In this section, we use the properties of width established above to prove amplitude inequalities and derive applications to DG-modules. We begin with a dual intersection theorem (\Cref{thm:4.6dual}), which yields an amplitude inequality for the derived completion of derived Hom complexes (\Cref{thm:4.7dual}). We then use duality to obtain a corresponding inequality for derived tensor products (\Cref{cor:ampLLam}) and deduce a DG version of Iversen’s inequality (\Cref{cor:iversen}). Finally, we give a criterion for a DG-module to be an Auslander DG-module (\Cref{def:auslander} and \Cref{cor:auslander}) and establish a DG version of Bass’ conjecture (\Cref{cor:bass}), treating both DG-modules with finitely generated homology and derived $\m$-complete DG-modules.

We start with one of the main results of this section, which provides a dual to the intersection theorem proved in \cite[Theorem 4.6]{FN}.
\begin{theorem}\label{thm:4.6dual}
Let $X\in\Dfb$ with $X\not\simeq0$ and $Y\in\Dba$ with $\width_RY<\infty$ and $\injdim_RY<\infty$. Then
\[
\sup\LLam[\m]{(\RHom_R(X,Y))}\geq\ld_RX+\width_RY-\depth R.
\]
\end{theorem}
\begin{proof}
We first note that $Y$ is bounded by \cite[Lemma 2.1(2)]{yang}.
The asserted inequality follows from the string of (in)equalities below
\begin{align*}
\sup\LL[\m]{R}{(\RHom_R(X,Y))}&=\sup\RHom_R(\HR,\LL[\m]{R}{(\RHom_R(X,Y))})\\
&=\sup\LL[\m]{\HR}{(\RHom_R(\HR,\RHom_R(X,Y)))}\\
&=\sup\LL[\m]{\HR}{(\RHom_{\HR}(\HR\lotimes_RX,\RHom_R(\HR,Y)))}\\
&\geq\dim_{\HR} X\lotimes_R\HR+\width_{\HR}\RHom_R(\HR,Y)-\depth\HR\\
&=\ld_RX+\width_RY-\depth R+\depth \HR-\depth\HR\\
&=\ld_RX+\width_RY-\depth R
\end{align*}
where the first equality follows from \cite[(1.2)]{ShaulCM} since $\LL[\m]{R}{(\RHom_R(X,Y))}\in\Dba$ because $\RHom_R(X, Y) \in \Dba$ by \cite[Lemma 3.2]{ShaulFinDim}, and $\LL[\m]{R}{(\RHom_R(X,Y))} \simeq \Hom_{R_0}(\Tel(R_0; \mathbf{x}), \RHom_R(X, Y))$ for a suitable lift $\mathbf{x}$ of generators of $\m$ by \cite[Proposition 2.4]{ShaulCompletion}. Since $\Tel(R_0, \mathbf{x})$ is bounded, another application of \cite[Lemma 3.2]{ShaulFinDim} shows that $\LL[\m]{R}{(\RHom_R(X,Y))}\in\Dba$. The second equality follows from \Cref{thm:LLamCobaseChange}.
The third equality is simply derived tensor-Hom adjunction, and the fourth equality follows from \Cref{prp:widthCobaseChange} and \cite[Proposition 2.13]{ShaulCM}. 

It remains to justify the inequality. Set $U=\HR\lotimes_RX$ and $V=\RHom_R(\HR,Y)$. First we notice that $U\in\Dfbb[\HR]$. Moreover, since $V\in\Dba[\HR]$ and has finite injective dimension by \cite[Corollary 2.3(iii)]{ShaulFinDim}, it follows that $V\in\Db$. Moreover $V$ has finite width by \Cref{prp:widthCobaseChange}. Therefore $\RHom_{\HR}(\kk,V)$ is bounded and nonzero by \Cref{prp:finiteDepth}. For each nonzero $\H_i(U)$ we can apply the right inequality in \cite[($*$) in the proof of Theorem 2.3]{amp} to get
\[
\sup\LL[\m]{\HR}{(\RHom_{\HR}(\H_i(U),V))}
\geq\dim_{\HR}\H_i(U)+\width_{\HR}V-\depth\HR.
\]
Since $\LL[\m]{\HR}{(V)}$ is bounded above, we can apply \cite[Proposition 17.6.14]{LarsBook} in the computation below
\begin{align*}
&\sup\LL[\m]{\HR}{(\RHom_{\HR}(U,V))}\\
&\quad=\sup\RHom_{\HR}(U,\LL[\m]{\HR}{(V)})\\
&\quad=\sup_{i\in\mathbb Z}
\bigl\{\sup\RHom_{\HR}(\HH{i}{U},\LL[\m]{\HR}{(V)})-i\bigr\}\\
&\quad=\sup_{i\in\mathbb Z}
\bigl\{\sup\LL[\m]{\HR}{(\RHom_{\HR}(\HH{i}{U},V))}-i\bigr\}\\
&\quad\geq\sup_{i\in\mathbb Z}
\bigl\{\dim_{\HR}\HH{i}{U}-i\bigr\}
+\width_{\HR}V-\depth\HR\\
&\quad=\dim_{\HR}U+\width_{\HR}V-\depth\HR.
\end{align*}
This is the desired inequality.
\end{proof}

The next result is a DG version of the right inequality in \cite[Theorem 2.3]{amp} and it is a dual statement to \cite[Corollary 4.7]{FN}. It generalizes \cite[Proposition 4.5]{yang} in view of \Cref{prp:ampCompletion}.
\begin{theorem}\label{thm:4.7dual}
Let $X\in\Dfba$ with $X\not\simeq0$ and $Y\in\Db$ with $\width_RY<\infty$ and $\injdim_RY<\infty$. Then
\[
\amp\LLam[\m]{(\RHom_R(X,Y))}\geq\amp\RGam[\m](X).
\]
\end{theorem}
\begin{proof}
Let $Z\colonequals\LLam[\m]{(\RHom_R(X,Y))}$. We first show that $Z\not\simeq0$. Since $Y$ is bounded and of finite injective dimension, it follows that $\RHom_R(\kk,Y)$ is bounded as well. By \Cref{prp:finiteDepth} $\depth_R Y=-\sup\RHom_R(\kk,Y)$ is finite, and therefore $\RHom_R(\kk,Y)\not\simeq0$. We note that since $\kk$ is derived $\m$-torsion, one has the following isomorphism
\begin{equation}\label{eq:Z}
\RHom_R(\kk,Z)\simeq\RHom_\kk(\kk\lotimes_RX,\RHom_R(\kk,Y)).
\end{equation}
Let $K$ be the Koszul complex on a finite list of generators of $\m$.  Since $K\in\thick\kk$ it follows that $K\lotimes_RX\in\thick(\kk\lotimes_RX)$. If $\kk\lotimes_RX\simeq0$, then $K\lotimes_RX\simeq0$ and therefore $X\simeq0$. This shows that $\kk\lotimes_RX\not\simeq0$, and therefore by \eqref{eq:Z} it follows that $Z\not\simeq0$. We can assume that $Z\in\Db$, otherwise the desired inequality is obvious. 

Next we show that under these hypotheses one has $X\in\Dfb$. Since $Z\in\Db$, one has $\RHom_R(\kk,Z)\in\Dba$, and by \eqref{eq:Z} one has the following equality
\[
\sup\RHom_R(\kk,Z)=\sup\RHom_R(\kk,Y)-\inf \kk\lotimes_RX,
\]
and therefore $\kk\lotimes_R X\in \Dbb$. Since $K\lotimes_RX\in\thick(\kk\lotimes_RX)$, it follows that $K\lotimes_RX\in\Dbb$, and by Nakayama $X\in\Dbb$.

The reduction to $X\in\Dfb$ will allow us to use \Cref{thm:4.6dual} in the string of (in)equalities below
\begin{align*}
\amp\LLam[\m]{(\RHom_R(X,Y))}&=\sup\LLam[\m]{(\RHom_R(X,Y))}-\inf\LLam[\m]{(\RHom_R(X,Y))}\\
&=\sup\LLam[\m]{(\RHom_R(X,Y))}-\width_R\RHom_R(X,Y)\\
&=\sup\LLam[\m]{(\RHom_R(X,Y))}-\depth_RX-\width_RY+\depth R\\
&\geq\ld_RX+\width_RY-\depth R-\depth_RX-\width_RY+\depth R\\
&=\ld_RX-\depth_RX\\
&=\amp\RGam[\m]{(X)},
\end{align*}
where the second equality follows from \Cref{thm:width} since $\RHom_R(X,Y)\in\Dbb$ (this holds as $X \in \Db[R]$ and since $Y$ has finite injective dimension, see \cite[Definition 2.1]{ShaulFinDim}), the third equality from \Cref{thm:dwf}, the inequality from \Cref{thm:4.6dual}.
\end{proof}

We now provide a series of preliminary lemmas that will be used to prove the main results of this section. The next result establishes the tensor-evaluation isomorphism of \cite[Lemma 2.7]{ShaulFinDim} under different hypotheses.

\begin{lemma}\label{lem:tensoreval}
Let $M \in \Db$ with $\projdim_RM < \infty$, $X \in \Db$ with $\flatdim_RX < \infty$, and $N \in \Df$. Then the derived tensor evaluation morphism
\[
\tau: \RHom_R(M, X) \lotimes_R N \to \RHom_R(M, X \lotimes_R N)
\]
is an isomorphism.
\end{lemma}
\begin{proof}
If $M\simeq0$ or $X\simeq0$, the assertion is immediate, so we may
assume that both are nonzero.
Consider the triangulated functors $F\colon\RHom_R(M,X)\lotimes_R-$ and
$G\colon\RHom_R(M,X\lotimes_R-)$. The functor $G$ is homologically
bounded by the assumptions on $M$ and $X$.
(This follows from the same logic as \cite[Proposition A.30c]{LarsBook}
and \cite[Proposition A.31c]{LarsBook}).
Here and below, homological boundedness of a functor means finite
homological displacement, in the sense of \cite[Definition 2.1]{tilting}
with homological indexing. More explicitly, the same arguments give
\begin{align*}
\inf G(L)&\geq\inf X+\inf L-\projdim_RM
&&\text{for }L\in\Dbb,\\
\sup G(L)&\leq\flatdim_RX+\sup L-\inf M
&&\text{for }L\in\Dba.
\end{align*}
Thus these displacement bounds are uniform in $L$.

By the boundedness assumptions on $M$ and $X$ and by the assumption
that $\projdim_RM<\infty$, we have that $\RHom_R(M,X)\in\Db$.
We now show that $\flatdim_R\RHom_R(M,X)<\infty$.
We establish a uniform upper bound for
$\sup(\RHom_R(M,X)\lotimes_R\HR/I)$ as $I$ ranges over the ideals of
$\HR$.
Since $\HR/I$ is in $\Dfb$, we have that
$\RHom_R(M,X)\lotimes_R\HR/I\simeq\RHom_R(M,X\lotimes_R\HR/I)$
by \cite[Lemma 2.7(iv)]{ShaulFinDim}. It follows that
\begin{align*}
\sup\bigl(\RHom_R(M,X)\lotimes_R\HR/I\bigr)
&=\sup\RHom_R(M,X\lotimes_R\HR/I)\\
&\leq\sup(X\lotimes_R\HR/I)-\inf M\\
&\leq\flatdim_RX-\inf M.
\end{align*}
The first inequality follows from \cite[Lemma 3.2(i)]{ShaulFinDim},
and the second from the definition of flat dimension.
The upper bound is independent of $I$.
Applying \cite[Theorem 8.3.11]{LarsBook} to
$\HR\lotimes_R\RHom_R(M,X)$, and using
\cite[Corollary 2.3(ii)]{ShaulFinDim}, therefore gives
\[
\flatdim_R\RHom_R(M,X)\leq\flatdim_RX-\inf M<\infty.
\]
Since $\RHom_R(M,X)$ is homologically bounded and has finite flat
dimension, the functor $F$ is homologically bounded.
(This follows from the same logic as \cite[Proposition A.31c]{LarsBook}).
The same estimates for derived tensor products give
\begin{align*}
\inf F(L)&\geq\inf\RHom_R(M,X)+\inf L
&&\text{for }L\in\Dbb,\\
\sup F(L)&\leq\flatdim_R\RHom_R(M,X)+\sup L
&&\text{for }L\in\Dba.
\end{align*}
Hence both $F$ and $G$ have finite homological displacement.
The tensor-evaluation morphisms define a natural transformation
$\tau:F\to G$, and its component at $R$ is the canonical isomorphism
\[
\tau_R:\RHom_R(M,X)\lotimes_RR
\xrightarrow{\ \simeq\ }\RHom_R(M,X\lotimes_RR).
\]
Thus, applying \cite[Theorem 2.11(2)]{tilting} to the functors $F$ and
$G$ yields that $\tau$ is an isomorphism.
\end{proof}
The next result should be compared to \cite[Lemma 3.5]{FN}, and is a DG version of \cite[Proposition 2.7]{amp}.
\begin{lemma}\label{lem:DCtransfer}
Let $X \in \Df$, let $Y \in \Db$ with $\flatdim_RY < \infty$, and let $I$ be an ideal in $\HR$. Then we have the following isomorphism:
\[
\LLam[I]{(X \lotimes_R Y)} \simeq X \lotimes_R \LLam[I]{(Y)}.
\]
\end{lemma}

\begin{proof}
Let $J$ be a finitely generated ideal in $R_0$ generated by lifts of the generators of $I$, and denote a finite list of generators of $J$ as $\mathbf{j}$. We have an isomorphism $\LLam[I]{(X \lotimes_R Y)} \simeq \RHom_{R_0}(\check{C}(\mathbf{j}), X \lotimes_R Y)$ in $\D(R)$ by \cite[Proposition 2.4]{ShaulCompletion}. The DG $R$-module $\RHom_{R_0}(\check{C}(\mathbf{j}), X \lotimes_R Y)$ is isomorphic to $\RHom_R(\check{C}(\mathbf{j}) \lotimes_{R_0} R, X \lotimes_R Y)$ through derived Hom-tensor adjunction. Note that the DG $R$-module $\check{C}(\mathbf{j}) \lotimes_{R_0} R$ has finite projective dimension, as $\projdim_{R_0} \check{C}(\mathbf{j})<\infty$. (This follows from \cite[Corollary 2.3]{ShaulFinDim} and \cite[Proposition 8.1.4]{LarsBook}). Thus, we have the isomorphisms

\begin{align*}
\LLam[I]{(X \lotimes_R Y)}  &\simeq \RHom_R(\check{C}(\mathbf{j}) \lotimes_{R_0} R, X \lotimes_R Y) \\
                            &\simeq X \lotimes_R \RHom_R(\check{C}(\mathbf{j}) \lotimes_{R_0} R, Y) \\
                            &\simeq X \lotimes_R \LLam[I]{(Y)}
\end{align*}
where the second isomorphism follows from \Cref{lem:tensoreval} and the third isomorphism from \cite[Proposition 2.4]{ShaulCompletion}.
\end{proof}
The next result is a DG version of \cite[Theorem 13.2.5]{LarsBook}, and will also be used in the proofs of the results that follow.
\begin{corollary}
\label{cor:DCFG}
Let $M \in \Df$. Then we have the following isomorphism:
\[
\LLam{(M)} \simeq M \lotimes_R \widehat{R}.
\]
\end{corollary}

\begin{remark}
\Cref{prp:ampCompletion} also follows from \Cref{cor:DCFG} along with \cite[Lemma 4.8]{ShaulFinDim} and the proof of \cite[Theorem 4.1]{ShaulFinDim}.
\end{remark}

The next lemma shows that the depth of a derived completion is the same whether computed over $R$ or $\wh{R}$. Together with \Cref{prp:depthLLam}, it will allow us to transfer the finite-depth hypothesis to the completed DG-ring in the proof of \Cref{cor:ampLLam}.

\begin{lemma}\label{lem:depthLLamhat}
Let $X\in\Db$, then
\[
\depth_R\LLam[\m]{(X)}=\depth_{\wh R}\LLamhat[\m]{(X)}.
\]
\end{lemma}
\begin{proof}
By
\cite[Proposition 5.4]{ShaulCompletion}, we have an isomorphism
\[
\LLamhat[\m]{(X)}
\simeq
\RHom_R(\wh R,\LLam[\m]{(X)})
\]
in $\D(\wh R)$. Since $R$ and $\wh R$ have the same residue field,
we have the chain of equalities
\begin{align*}
\depth_{\wh R}\LLamhat[\m]{(X)}
&=-\sup\RHom_{\wh R}(\kk,\LLamhat[\m]{(X)})\\
&=-\sup\RHom_{\wh R}
  (\kk,\RHom_R(\wh R,\LLam[\m]{(X)}))\\
&=-\sup\RHom_R
  (\kk\lotimes_{\wh R}\wh R,\LLam[\m]{(X)})\\
&=-\sup\RHom_R(\kk,\LLam[\m]{(X)})\\
&=\depth_R\LLam[\m]{(X)}.
\end{align*}
The first and last equalities follow from the definition of depth,
the second from the isomorphism above, the third from derived
Hom-tensor adjunction, and the fourth from tensor cancellation.
\end{proof}

The next result is a DG version of \cite[Proposition 15.4.30(a)]{LarsBook}.

\begin{lemma}\label{lem:flatDim} 
Let $X\in\Db$, then
\begin{enumerate}
\item $\flatdim_R\LLam[\m]{(X)}\leq\flatdim_RX$.
\item $\flatdim_{\wh R}\LLamhat{(X)}\leq\flatdim_RX$.
\end{enumerate}
\end{lemma}
\begin{proof}
We may assume that $\flatdim_RX<\infty$, as otherwise the
inequalities are immediate. We have that
$\HR\lotimes_RX\in\Db[\HR]$ and
\[
\flatdim_{\HR}(\HR\lotimes_RX)=\flatdim_RX
\]
by \cite[Corollary 2.3]{ShaulFinDim}. Moreover, $\LLam[\m]{(X)}\in\Db$; indeed,
\[
\inf\LLam[\m]{(X)}=\width_RX\geq\inf X>-\infty
\]
by \Cref{thm:width,thm:widthInf}, while
\[
\sup\LLam[\m]{(X)}
\leq\ld R-\depth_RX
\leq\ld R+\sup X<\infty
\]
by \Cref{thm:bounds} and the inequality
$\depth_RX\geq-\sup X$.

$(1)$. This is just a special case of \cite[Theorem 4.3]{ShaulHomDim}.

$(2)$. We have that
$\H_0(\wh R)\simeq\wh{\HR}$ by
\cite[Proposition 1.7]{ShaulCM}. Since the completion
map $R\to\wh R$ has flat dimension zero by
\cite[Corollary 4.6]{ShaulHomDim}, we have
\[
\HR\lotimes_R\wh R\simeq\wh{\HR}.
\]
Moreover, $\projdim_R\wh R<\infty$ by
\cite[Theorem 4.1]{ShaulFinDim}. Since
$\LLam[\m]{(X)}\in\Db$ and
$\flatdim_R\LLam[\m]{(X)}<\infty$ by part $(1)$,
we can apply \Cref{lem:tensoreval} to obtain the
following isomorphisms in $\D(\wh{\HR})$:
\begin{align*}
\wh{\HR}\lotimes_{\wh R}\LLamhat[\m]{(X)}
&\simeq
\HR\lotimes_R\RHom_R(\wh R,\LLam[\m]{(X)})\\
&\simeq
\RHom_R(\wh R,\HR\lotimes_R\LLam[\m]{(X)})\\
&\simeq
\RHom_{\HR}
(\HR\lotimes_R\wh R,\HR\lotimes_R\LLam[\m]{(X)})\\
&\simeq
\RHom_{\HR}
(\wh{\HR},\LLam[\m\HR]{(\HR\lotimes_RX)})\\
&\simeq
\LLamhat[\m\HR]{(\HR\lotimes_RX)}.
\end{align*}
The first isomorphism follows from tensor cancellation
and \cite[Proposition 5.4]{ShaulCompletion}, the second
from \Cref{lem:tensoreval}, and the third from derived
Hom-tensor adjunction. The fourth follows from
\cite[Lemma 4.2]{ShaulHomDim} and the isomorphism
$\HR\lotimes_R\wh R\simeq\wh{\HR}$, while the last
follows from \cite[Proposition 5.4]{ShaulCompletion}
applied to $\HR$.

Thus, we have the (in)equalities
\begin{align*}
\flatdim_{\wh R}\LLamhat[\m]{(X)}
&=\flatdim_{\wh{\HR}}
  (\wh{\HR}\lotimes_{\wh R}\LLamhat[\m]{(X)})\\
&=\flatdim_{\wh{\HR}}
  \LLamhat[\m\HR]{(\HR\lotimes_RX)}\\
&=\flatdim_{\HR}
  \LLam[\m\HR]{(\HR\lotimes_RX)}\\
&\leq\flatdim_RX.
\end{align*}
The first equality follows from
\cite[Corollary 2.3]{ShaulFinDim}, the second from
the isomorphisms above, and the third from
\cite[Proposition 1.2(f)]{amp}. The inequality follows from part (1), \Cref{lem:4.2}, and \cite[Corollary 2.3(ii)]{ShaulFinDim}.
\end{proof}

The following is one of the main results of this section and provides a DG version of the right inequality in \cite[Theorem 2.8]{amp}.
\begin{theorem}\label{cor:ampLLam}
Let $Y\in\Db$ with $\depth_RY<\infty$ and $\flatdim_RY<\infty$, and let $X\in\Df$. We have the following inequality
\[
\amp\LLam[\m]{(X\lotimes_RY})\geq\amp X.
\]
\end{theorem}

\begin{proof}
We may assume $X\not\simeq0$. Let $D$ be a dualizing DG-module for $\wh{R}$ (which exists by \cite[Proposition 7.21]{injective}), and shift $D$ so that it is normalized. Finally, set $\wh{X} = X \lotimes_R \wh{R}$. 

We first show that $\LLam[\m]{(X\lotimes_RY)}\not\simeq0$. We observe that for every $M\in\D(R)$ one has
\begin{align*}
\kk\lotimes_R\LLam[\m]{(M)}&\simeq\RGam[\m]{(\kk)}\lotimes_R\LLam[\m]{(M)}\\
&\simeq\kk\lotimes_R\RGam[\m]{(\LLam[\m]{(M)})}\\
&\simeq\kk\lotimes_R\RGam[\m]{(M)}\\
&\simeq\RGam[\m]{(\kk)}\lotimes_RM\\
&\simeq\kk\lotimes_RM.
\end{align*}
We note that $\kk\lotimes_RY\not\simeq0$ since $\width_RY<\infty$ (by \Cref{prp:finiteDepth}). Moreover $\kk\lotimes_RY\in\Db$ since $Y\in\Db$ and $\flatdim_RY<\infty$. We note that as in the proof of \Cref{thm:4.7dual} one can argue that $\kk\lotimes_RX\not\simeq0$. Therefore, it follows from 
\begin{equation}\label{eq:Z2}
\kk\lotimes_R\LLam[\m]{(X\lotimes_RY)}\simeq\kk\lotimes_RX\lotimes_RY\simeq(\kk\lotimes_R X)\lotimes_\kk(\kk\lotimes_RY)
\end{equation}
that $\LLam[\m]{(X\lotimes_RY)}\not\simeq0$.

Thus we can assume $\LLam[\m]{(X\lotimes_RY)}\in\Db$ since otherwise the desired inequality is trivial. We show that under these hypotheses one has $X\in\Dbb$, this will imply that $\RHom_{\wh R}(\wh{X},D)\in\Dfba[\wh R]$, which will allow us to use \Cref{thm:4.7dual} in the computation below. Since $\LLam[\m]{(X\lotimes_RY)}\in\Db$, it follows that $\kk\lotimes_R\LLam[\m]{(X\lotimes_RY)}\in\Dbb$, moreover from \eqref{eq:Z2} we deduce
\[
\inf(\kk\lotimes_R\LLam[\m]{(X\lotimes_RY)})=\inf(\kk\lotimes_RX)+\inf(\kk\lotimes_RY),
\]
and therefore $\kk\lotimes_RX\in\Dbb$. Now one can deduce $X\in\Dbb$ as in the proof of \Cref{thm:4.7dual}. Since $R\to\wh R$ is faithfully flat, we have
$\wh X\in\Dfbb[\wh R]$ and $\wh X\not\simeq0$.
Thus \cite[Proposition 7.2]{tilting} gives
$\RHom_{\wh R}(\wh X,D)\in\Dfba[\wh R]$, and biduality shows
that it is nonzero.

We now observe the following chain of isomorphisms:
\begin{align*}
\LLamhat[\m]{(X\lotimes_RY)}&\simeq\LLam[\wh{\m}]{(\LLamhat[\m]{(X\lotimes_RY)})}\\
&\simeq\LLam[\wh{\m}]{(\RHom_R(\wh{R}, \LLam{(X \lotimes_R Y)}))}\\
&\simeq\LLam[\wh{\m}]{(\RHom_R(\wh{R}, X\lotimes_R\LLam[\m]{(Y)}))}\\
&\simeq\LLam[\wh{\m}]{(\RHom_R(\wh{R}, \LLam[\m]{(Y)})\lotimes_R X)}\\
&\simeq\LLam[\wh{\m}]{(\RHom_R(\wh{R}, \LLam[\m]{(Y)}) \lotimes_{\wh{R}} \wh{X})}\\
&\simeq\LLam[\wh{\m}]{(\LLamhat[\m]{(Y)} \lotimes_{\wh{R}} \wh{X})}\\
&\simeq\LLam[\wh{\m}]{(\LLamhat[\m]{(Y)} \lotimes_{\wh{R}} (\RHom_{\wh{R}}(\RHom_{\wh{R}}(\wh{X},D),D)))}\\
&\simeq\LLam[\wh{\m}]{(\RHom_{\wh{R}}(\RHom_{\wh{R}}(\wh{X},D),D\lotimes_{\wh{R}}\LLamhat[\m]{(Y)}))}.
\end{align*}
The first isomorphism is \cite[Lemma A.4]{injective}, and the second isomorphism is \cite[Proposition 5.4]{ShaulCompletion}. The third isomorphism follows from Lemma~\ref{lem:DCtransfer}, and the fourth isomorphism is derived tensor evaluation (Lemma~\ref{lem:tensoreval}), noting that $\wh{R}$ has finite projective dimension by \cite[Corollary 4.6]{ShaulHomDim} and \cite[Theorem 4.1]{ShaulFinDim}, and $\LLam Y$ has finite flat dimension by \Cref{lem:flatDim}. The fifth isomorphism is derived tensor cancellation, and the sixth isomorphism is \cite[Proposition 5.4]{ShaulCompletion}. The seventh isomorphism holds by Grothendieck duality \cite[Proposition 7.2]{tilting}.

The DG-module $D \lotimes_{\wh{R}} \LLamhat[\m]{(Y)}$ has finite injective dimension over $\wh{R}$ by \cite[Proposition 3.6]{ShaulFinDim}, and is bounded as both $D$ and $\LLamhat[\m]{(Y)}$ are bounded and $\LLamhat[\m]{(Y)}$ has finite flat dimension. Boundedness of $\LLamhat[\m]{(Y)}$ follows from
boundedness of $\LLam[\m]{(Y)}$ and
\cite[Proposition 5.5]{ShaulCompletion}.

To justify the eighth isomorphism, first note that $D\in\Dfb[\wh R]$
and $\LLamhat[\m]{(Y)}$ have finite flat dimension over $\wh R$
by \Cref{lem:flatDim}. Thus tensor evaluation
\cite[Lemma 2.7(iii)]{ShaulFinDim}, applied with arguments
$D$, $D$, and $\LLamhat[\m]{(Y)}$, gives
\[
\LLamhat[\m]{(Y)}
\simeq \RHom_{\wh R}(D,D)\lotimes_{\wh R}\LLamhat[\m]{(Y)}
\simeq \RHom_{\wh R}(D,D\lotimes_{\wh R}\LLamhat[\m]{(Y)}).
\]
Since $\wh X\in\Dfbb[\wh R]$, $D$ is bounded, and
$D\lotimes_{\wh R}\LLamhat[\m]{(Y)}$ has finite injective dimension, Hom-evaluation (part~(i) of \cite[Lemma 2.7]{ShaulFinDim}) applies with arguments $\wh X$, $D$, and
$D\lotimes_{\wh R}\LLamhat[\m]{(Y)}$. Consequently,
\begin{align*}
\wh X\lotimes_{\wh R}\LLamhat[\m]{(Y)}
&\simeq\wh X\lotimes_{\wh R}
\RHom_{\wh R}(D,D\lotimes_{\wh R}\LLamhat[\m]{(Y)})\\
&\simeq\RHom_{\wh R}(\RHom_{\wh R}(\wh X,D),
D\lotimes_{\wh R}\LLamhat[\m]{(Y)}).
\end{align*}
Together with biduality for $\wh X$ and commutativity of the derived
tensor product, this gives the eighth isomorphism after applying
$\LLam[\wh\m]{(-)}$ over $\wh R$.

In addition, the width of $D \lotimes_{\wh{R}} \LLamhat[\m]{Y}$ is finite by \cite[Lemma 2.5]{yang}, \Cref{prp:finiteDepth}, \Cref{prp:depthLLam} and \Cref{lem:depthLLamhat}. Thus we can apply \Cref{thm:4.7dual} over $\wh R$ to obtain the following inequality
\begin{align*}
    \amp\LLamhat[\m]{(X\lotimes_RY)}
    &= \amp\LLam[\wh{\m}]{(\RHom_{\wh{R}}(\RHom_{\wh{R}}(\wh{X},D),D\lotimes_{\wh{R}}\LLamhat[\m]{(Y)}))} \\
    &\geq \amp\RGam[\wh\m]{(\RHom_{\wh R}(\wh{X}, D))} \\
    &= \amp\RHom_{\wh R}(\wh{X}, E_{\wh R}) \\
    &= \amp\wh{X}
\end{align*}
where the second equality follows from \cite[Propositions 7.24 and 7.25]{injective} over $\wh R$ and the third equality follows from \cite[Theorem 4.10]{injective}.
Applying the forgetful functor to the DG-module $\LLamhat[\m]{(X \lotimes_R Y)}$ as described in \cite[Proposition 5.5]{ShaulCompletion} yields that $\amp\LLamhat[\m]{(X\lotimes_RY)} = \amp\LLam[\m]{(X \lotimes_R Y)}$. We also have that $\amp\wh{X} = \amp X$ by \cite[Lemma 4.8]{ShaulFinDim}, which yields the necessary inequality
\[
\amp\LLam[\m]{(X \lotimes_R Y)} \geq \amp X.\qedhere
\]
\end{proof}

As a consequence of the previous theorem we prove a DG version of Iversen's inequality proved in \cite[Corollary 2.9]{amp}. In the DG setting it was already proved for DG-modules with finitely generated homology by Yang in \cite[Proposition 4.1]{yang}. We reprove it and show that it also holds for derived $\m$-complete DG-modules.

\begin{corollary}\label{cor:iversen}
Let $Y\in\Db$ with $Y\not\simeq0$ and $\flatdim_RY<\infty$. Let $X\in\Df$. If $Y\in\Dfb$ or $Y$ is derived $\m$-complete, then
\[
\amp X\lotimes_RY\geq\amp X.
\]
\end{corollary}
\begin{proof}
We can assume that $X\lotimes_RY\in\Db$, otherwise the inequality is trivial. In either case, \Cref{thm:widthInf} gives
$\width_RY=\inf Y<\infty$, since $Y\not\simeq0$ is bounded.
Thus $\depth_RY<\infty$ by \Cref{prp:finiteDepth}.

If $Y$ is derived $\m$-complete, then so is $X\lotimes_RY$ by \Cref{lem:DCtransfer}. Now one concludes by invoking \Cref{cor:ampLLam}.

If $Y\in\Df$, then $Y$ is perfect.
Indeed, \cite[Theorem 4.1]{ShaulFinDim} gives
$\projdim_RY<\infty$, and $Y\in\Dfb$ is therefore compact
by the proof of \cite[Lemma 2.7]{ShaulFinDim}.
Hence $Y\in\thick_R R$; see
\cite[Theorem 5.20 and Corollary 5.21]{tilting}.
It follows that
\[
X\lotimes_RY\in\thick_RX\subseteq\Df,
\]
where the inclusion holds because $\Df$ is a thick subcategory
of $\D(R)$. Thus $X\lotimes_RY\in\Dfb$, so by \Cref{prp:ampCompletion} 
\[
\amp X\lotimes_RY=\amp\LLam[\m]{(X\lotimes_RY)}\geq\amp X,
\]
where the inequality follows from \Cref{cor:ampLLam}.
\end{proof}

We recall the definition of regular sequences from \cite[Definition 5.2]{ShaulCM}. Let $0\not\simeq M\in\Dba$. An element $x\in\m$ is called $M$-regular if it is $\H_{\sup M}(M)$-regular. Inductively, a sequence $x_1,\ldots,x_n\in\m$ is called $M$-regular if $x_1$ is $M$-regular and the sequence $x_2,\ldots,x_n$ is $K(x_1)\lotimes_RM$-regular.

\begin{remark}\label{rmk:regular}
It is noted in \cite[(5.6)]{ShaulCM} that by Nakayama's lemma if $0\not\simeq M\in\Dfb$, then a sequence $\x\subseteq\fm$ is $M$-regular if and only if $\amp M=\amp K(\x)\lotimes_RM$. Since Nakayama's lemma also holds for derived $\m$-complete DG-modules, see \cite[Lemma 3.3]{Zach}, it follows that if $M\not\simeq0$ is a derived $\m$-complete bounded DG $R$-module, then $\x$ is $M$-regular if and only if $\amp M=\amp K(\x)\lotimes_RM$.
\end{remark}

Inspired by \cite{auslander}, we make the following
\begin{definition}\label{def:auslander}
Let $0\not\simeq Y\in\Dba$. We say that $Y$ is an \emph{Auslander DG R-module} if every $Y$-regular sequence is also $R$-regular.
\end{definition}

As an application of Iversen's inequality proved above, we generalize \cite[Example 2.10]{amp} to the DG setting. For DG-modules with finitely generated homology it was already proved in \cite[Remark 4.2(2)]{yang}.

\begin{corollary}\label{cor:auslander}
Let $Y\in\Db$ with $\flatdim_RY<\infty$ and $\amp Y=\amp R$. If $Y\in\Dfb$ or if $Y$ is derived $\m$-complete, then $Y$ is an Auslander DG $R$-module.
\end{corollary}
\begin{proof}
Let $\x$ be a $Y$-regular sequence, then
\begin{align*}
\amp R&=\amp Y\\
&=\amp K(\x)\lotimes_RY\\
&\geq\amp K(\x)
\end{align*}
where the first equality holds by our hypothesis, the second follows from \Cref{rmk:regular} and the inequality from \Cref{cor:iversen}. The reverse inequality follows from \cite[Lemma 3.13]{minamoto} and therefore $\amp K(\x)=\amp R$, which implies that $\x$ is also an $R$-regular sequence by another application of \Cref{rmk:regular}.
\end{proof}

As a final application of the results developed in this section, we show that the DG version of Bass' conjecture, known as Minamoto's conjecture or Shaul's conjecture, holds for DG-modules with finitely generated homology, which was already proved in \cite[Proposition 4.5]{yang}, and for derived $\m$-complete DG-modules. We recall that a DG-ring $R$ is called \emph{Cohen-Macaulay} if $\amp R=\amp\RGam[\m]{(R)}$.

\begin{corollary}\label{cor:bass}
The following are equivalent
\begin{enumerate}
\item $R$ is Cohen-Macaulay.
\item There is a $Y\in\Dfb$ with $\injdim_RY<\infty$ and $\amp Y=\amp R$.
\item There is a derived $\m$-complete DG $R$-module $Y\in\Db$ with $\injdim_RY<\infty$ and $\amp Y=\amp R$.
\end{enumerate}
\end{corollary}
\begin{proof}
$(1)\Rightarrow(2)$. This is proved in \cite[Theorem 5.22(1)]{ShaulCM}.

$(1)\Rightarrow(3)$. Let $Y=\RHom_R(\RGam[\m]{(R)},E)$. Since $\flatdim_R\RGam[\m]{(R)}<\infty$, it follows from \cite[Lemma 7.12]{ShaulFinDim} that $\injdim_RY<\infty$. Moreover,
\[
\amp Y=\amp\RGam[\m]{(R)}=\amp R
\]
where the first equality follows from \cite[Theorem 4.10]{injective} and the second follows from the definition of a Cohen-Macaulay DG-ring. In addition, we have isomorphisms 
\begin{align*}
\LLam[\m]{(Y)}&\simeq\RHom_R(\RGam[\m]{(R)},\RHom_R(\RGam[\m]{(R)},E))\\
&\simeq\RHom_R(\RGam[\m]{(R)}\lotimes_R\RGam[\m]{(R)},E)\\
&\simeq Y.
\end{align*}
The first isomorphism follows from Greenlees-May duality \cite[Proposition 2.11]{ShaulCompletion}. Thus, we have that $Y$ is derived $\m$-complete.

$(2),(3)\Rightarrow(1)$. We first notice that $\width_RY<\infty$ by \Cref{thm:widthInf}. By taking $X=R$ in \Cref{thm:4.7dual} we deduce that 
\[
\amp\LLam[\m]{(Y)}\geq\amp\RGam[\m]{(R)}.
\]
If $Y\in\Dfb$ then $\amp\LLam[\m]{(Y)}=\amp Y$ by \Cref{prp:ampCompletion}, and if $Y$ is derived $\m$-complete, then trivially $\amp\LLam[\m]{(Y)}=\amp Y$. Therefore the hypotheses yield $\amp R\geq\amp\RGam[\m]{(R)}$. The reverse inequality always holds, see \cite[Theorem 4.1]{ShaulCM}. Therefore $\amp R=\amp\RGam[\m]{(R)}$, i.e. $R$ is Cohen-Macaulay.
\end{proof}

\bibliographystyle{amsplain}
\bibliography{biblio}

@Book{LarsBook,
 Author = {Christensen, Lars Winther and Foxby, Hans-Bj{\o}rn and Holm, Henrik},
 Title = {Derived category methods in commutative algebra},
 FSeries = {Springer Monographs in Mathematics},
 Series = {Springer Monogr. Math.},
 ISSN = {1439-7382},
 ISBN = {978-3-031-77452-2; 978-3-031-77453-9},
 Year = {2024},
 Publisher = {Cham: Springer},
 Language = {English},
 DOI = {10.1007/978-3-031-77453-9},
 zbMATH = {7954534}
}

@article {ShaulCM,
    AUTHOR = {Shaul, Liran},
     TITLE = {The {C}ohen-{M}acaulay property in derived commutative
              algebra},
   JOURNAL = {Trans. Amer. Math. Soc.},
  FJOURNAL = {Transactions of the American Mathematical Society},
    VOLUME = {373},
      YEAR = {2020},
    NUMBER = {9},
     PAGES = {6095--6138},
      ISSN = {0002-9947,1088-6850},
   MRCLASS = {13C14 (13D45 16E35 16E45)},
  MRNUMBER = {4155173},
MRREVIEWER = {P.\ Schenzel},
       DOI = {10.1090/tran/8099},
       URL = {https://doi.org/10.1090/tran/8099},
}

@article {ShaulCompletion,
    AUTHOR = {Shaul, Liran},
     TITLE = {Completion and torsion over commutative {DG} rings},
   JOURNAL = {Israel J. Math.},
  FJOURNAL = {Israel Journal of Mathematics},
    VOLUME = {232},
      YEAR = {2019},
    NUMBER = {2},
     PAGES = {531--588},
      ISSN = {0021-2172,1565-8511},
   MRCLASS = {13D03 (13B35 16E45)},
  MRNUMBER = {3990952},
MRREVIEWER = {Keri\ Sather-Wagstaff},
       DOI = {10.1007/s11856-019-1866-6},
       URL = {https://doi.org/10.1007/s11856-019-1866-6},
}

@article {ShaulKoszul,
    AUTHOR = {Shaul, Liran},
     TITLE = {Koszul complexes over {C}ohen-{M}acaulay rings},
   JOURNAL = {Adv. Math.},
  FJOURNAL = {Advances in Mathematics},
    VOLUME = {386},
      YEAR = {2021},
     PAGES = {Paper No. 107806, 35},
      ISSN = {0001-8708,1090-2082},
   MRCLASS = {13H10 (13D09 16E45)},
  MRNUMBER = {4266748},
MRREVIEWER = {Xiaojun\ Chen},
       DOI = {10.1016/j.aim.2021.107806},
       URL = {https://doi.org/10.1016/j.aim.2021.107806},
}

@article {Zach,
    AUTHOR = {Nason, Zachary},
     TITLE = {Maximal {C}ohen-{M}acaulay {DG}-complexes},
   JOURNAL = {J. Pure Appl. Algebra},
  FJOURNAL = {Journal of Pure and Applied Algebra},
    VOLUME = {230},
      YEAR = {2026},
    NUMBER = {2},
     PAGES = {Paper No. 108193, 23},
      ISSN = {0022-4049,1873-1376},
   MRCLASS = {13C14 (13D09 16E35 16E45 16E65)},
  MRNUMBER = {5029991},
       DOI = {10.1016/j.jpaa.2026.108193},
       URL = {https://doi.org/10.1016/j.jpaa.2026.108193},
}

@article {ShaulFinDim,
    AUTHOR = {Bird, Isaac and Shaul, Liran and Sridhar, Prashanth and
              Williamson, Jordan},
     TITLE = {Finitistic dimensions over commutative {DG}-rings},
   JOURNAL = {Math. Z.},
  FJOURNAL = {Mathematische Zeitschrift},
    VOLUME = {309},
      YEAR = {2025},
    NUMBER = {1},
     PAGES = {Paper No. 3, 29},
      ISSN = {0025-5874,1432-1823},
   MRCLASS = {16E45 (13D03 13D09 18Gxx)},
  MRNUMBER = {4827125},
MRREVIEWER = {Jian\ Liu},
       DOI = {10.1007/s00209-024-03617-2},
       URL = {https://doi.org/10.1007/s00209-024-03617-2},
}

@article{amp,
  title={Amplitude inequalities for local (co)homology},
  author={Christensen, Lars Winther and Ferraro, Luigi and Holm, Henrik},
  journal={https://arxiv.org/abs/2606.02850},
  volume={},
  number={},
  pages={},
  year={},
  publisher={}
}

@article{yang, title={Intersection Theorem for {DG}-modules over non-positive {DG}-rings}, DOI={10.1017/prm.2025.10112}, journal={Proceedings of the Royal Society of Edinburgh: Section A Mathematics}, author={Yang, Xiaoyan}, year={2025}, pages={1–18}}

@article{minamoto,
  title={On commutative differential graded algebras},
  author={Minamoto, Hiroyuki },
  journal={https://arxiv.org/abs/1903.07514},
  volume={},
  number={},
  pages={},
  year={},
  publisher={}
}

@article {width,
    AUTHOR = {Rao, Yanping and Liu, Zhongkui and Yang, Xiaoyan and Chen,
              Wenjing},
     TITLE = {Depth and width for unbounded {DG}-modules},
   JOURNAL = {Algebra Colloq.},
  FJOURNAL = {Algebra Colloquium},
    VOLUME = {30},
      YEAR = {2023},
    NUMBER = {1},
     PAGES = {61--72},
      ISSN = {1005-3867,0219-1733},
   MRCLASS = {13C15 (13D30 18G10)},
  MRNUMBER = {4562914},
MRREVIEWER = {Liang\ Zhao},
       DOI = {10.1142/S1005386723000068},
       URL = {https://doi.org/10.1142/S1005386723000068},
}

@article {injective,
    AUTHOR = {Shaul, Liran},
     TITLE = {Injective {DG}-modules over non-positive {DG}-rings},
   JOURNAL = {J. Algebra},
  FJOURNAL = {Journal of Algebra},
    VOLUME = {515},
      YEAR = {2018},
     PAGES = {102--156},
      ISSN = {0021-8693,1090-266X},
   MRCLASS = {16D50 (13C11 13D45 16E45)},
  MRNUMBER = {3859962},
MRREVIEWER = {C.\ N\u ast\u asescu},
       DOI = {10.1016/j.jalgebra.2018.07.040},
       URL = {https://doi.org/10.1016/j.jalgebra.2018.07.040},
}

@book {strooker,
    AUTHOR = {Strooker, Jan R.},
     TITLE = {Homological questions in local algebra},
    SERIES = {London Mathematical Society Lecture Note Series},
    VOLUME = {145},
 PUBLISHER = {Cambridge University Press, Cambridge},
      YEAR = {1990},
     PAGES = {xiv+308},
      ISBN = {0-521-31526-3},
   MRCLASS = {13Hxx (13D03 13D40)},
  MRNUMBER = {1074178},
MRREVIEWER = {W.\ V.\ Vasconcelos},
       DOI = {10.1017/CBO9780511629242},
       URL = {https://doi.org/10.1017/CBO9780511629242},
}

@article{FN,
  title={Intersection theorems over {DG}-rings revisited},
  author={Ferraro, Luigi and Nason, Zach},
  journal={https://arxiv.org/abs/2606.32031},
  volume={},
  number={},
  pages={},
  year={},
  publisher={}
}

@article{ShaulHomDim,
	title = {Homological {Dimensions} of {Local} ({Co})homology {Over} {Commutative} {DG}-rings},
	volume = {61},
	copyright = {https://www.cambridge.org/core/terms},
	issn = {0008-4395, 1496-4287},
	url = {https://www.cambridge.org/core/product/identifier/S0008439500024279/type/journal_article},
	doi = {10.4153/CMB-2017-054-1},
	number = {4},
	urldate = {2026-06-15},
	journal = {Canadian Mathematical Bulletin},
	author = {Shaul, Liran},
	month = dec,
	year = {2018},
	pages = {865--877},
}

@article {auslander,
    AUTHOR = {Nasehpour, Peyman},
     TITLE = {Auslander modules},
   JOURNAL = {Beitr. Algebra Geom.},
  FJOURNAL = {Beitr\"age zur Algebra und Geometrie. Contributions to Algebra
              and Geometry},
    VOLUME = {59},
      YEAR = {2018},
    NUMBER = {4},
     PAGES = {617--624},
      ISSN = {0138-4821,2191-0383},
   MRCLASS = {13C05 (13A15)},
  MRNUMBER = {3871097},
MRREVIEWER = {Alveera\ Mehdi},
       DOI = {10.1007/s13366-018-0378-6},
       URL = {https://doi.org/10.1007/s13366-018-0378-6},
}

@article{tilting,
  title={Duality and tilting for commutative {DG} rings},
  author={Yekutieli, Amnon},
  journal={https://arxiv.org/abs/1312.6411},
  volume={},
  number={},
  pages={},
  year={},
  publisher={}
}

@article {ShaulLifting,
    AUTHOR = {Shaul, Liran and Williamson, Jordan},
     TITLE = {Lifting (co)stratifications between tensor triangulated
              categories},
   JOURNAL = {Israel J. Math.},
  FJOURNAL = {Israel Journal of Mathematics},
    VOLUME = {261},
      YEAR = {2024},
    NUMBER = {1},
     PAGES = {249--280},
      ISSN = {0021-2172,1565-8511},
   MRCLASS = {18G80 (18G35 18M05)},
  MRNUMBER = {4776493},
       DOI = {10.1007/s11856-023-2578-5},
       URL = {https://doi.org/10.1007/s11856-023-2578-5},
}
\end{document}